\documentclass[12pt]{amsart}

\usepackage{amsfonts}
\usepackage{amscd}
\usepackage{amssymb}

\allowdisplaybreaks

\usepackage{enumerate,enumitem,hyperref,cleveref}

\date{\today}

\newtheorem{thm}{Theorem}[section]
\newtheorem{cor}[thm]{Corollary}

\newtheorem{prop}[thm]{Proposition}
\theoremstyle{definition}

\theoremstyle{remark}
\newtheorem{rem}[thm]{Remark}

\numberwithin{equation}{section}

\newcommand{\R}{\mathbb R}

\newcommand{\He}{\mathbb H}

\newcommand{\C}{{\mathbb C}}

\renewcommand{\Re}{\operatorname{Re}}
\renewcommand{\Im}{\operatorname{Im}}
\newcommand{\tr}{\operatorname{tr}}

\usepackage{color}

\title[Representations of  Heisenberg groups]
{Algebras of entire functions and \\
 representations of the twisted Heisenberg group}

\author[ S. Thangavelu]{ Sundaram Thangavelu}

\address[S. Thangavelu]{Department of Mathematics, Indian Institute of Science, Bangalore--560012, India.}
\email{veluma@iisc.ac.in}

\begin{document}

\maketitle

\begin{center} \it{Dedicated to the memory of Professor K. R. Parthasarathy\\}
\end{center}
\vskip0.25in

\begin{abstract} On the twisted Fock spaces $ \mathcal{F}^\lambda(\C^{2n}) $ we consider a family of unitary operators $\rho_\lambda(a,b) $ indexed by $ (a,b) \in \C^n \times \C^n.$  The composition formula for $ \rho_\lambda(a,b) \circ \rho_\lambda(a^\prime,b^\prime) $ leads us to a group $ \He^n_\lambda(\C) $ which contains two copies of the Heisenberg group $ \He^n.$ The operators $ \rho_\lambda(a,b) $ lift to $ \He_\lambda^n(\C) $ providing  an irreducible unitary representation. However, its restriction to $ \He^n_\lambda(\R) $ is not irreducible.
 \end{abstract}


\section{Introduction} \label{Sec-intro}

We begin by recalling some well known facts about the Heisenberg groups $ \He^n $ and their representations. The group $  \He^n$ is  just $ \R^n \times \R^n \times \R $ (or equivalently $ \C^n \times \R $) equipped with the group law
$$ (x,y,t)(u,v,s) = (x+u, y+v, t+s+\frac{1}{2}(u \cdot y-x \cdot v)).$$
Note that $ \Re(\He^n) = \{ (x,0,t): x \in \R^n\, t \in \R\} $ and $ \Im(\He^n) = \{ (0,y,t): y \in \R^n, t \in \R \} $ are two subgroups of $ \He^n $ each being isomorphic to $ \R^{n+1}.$
For each $ \lambda \in \R, \lambda \neq 0,$ there is an irreducible unitary representation $ \pi_\lambda $ of $ \He^n$ on $ L^2(\R^n) $ explicitly given by
$$ \pi_\lambda(x,y,t)f(\xi) = e^{i\lambda t} e^{i\lambda(x \cdot \xi +\frac{1}{2} x\cdot y)} f(\xi+y),\,\,\, f \in L^2(\R^n).$$
By setting $ \pi(x,y) = \pi_1(x,y,0) $ we observe that the irreducibility of $ \pi_1 $ has the following consequence: if a bounded linear operator $ T $ on $ L^2(\R^n) $ commutes with $ \pi(x,y)$ for all $ x,y \in \R^n$ then $ T = cI $ for a constant $ c.$ However, there are non trivial operators that commute either with  the family  $ \pi(x,0) $ or with $ \pi(0,y).$\\

Observe that $ \pi(0,y)f(\xi) = f(\xi+y) $ are the translations and $ \pi(x,0)f(\xi) = e^{i x\cdot \xi}f(\xi) $ are the modulations. It is well known that any operator $ T $ that commutes with all translations $ \pi(0,y) $ is of the form $ T_m $ where
$$ T_mf(x) = (2\pi)^{-n/2} \int_{\R^n} e^{i x \cdot \xi}\, m(\xi)\,  \widehat{f}(\xi)\, d\xi $$
for some $ m \in L^\infty(\R^n).$ Similarly,  bounded linear operators $ S $ that commute with all modulations $ \pi(x,0) $ are of the form $ S_m f(x)= m(x) f(x) $ where $ m \in L^\infty(\R^n).$ Thus we see that though $ \pi_1 $ is  irreducible for $ \He^n $ its restriction to neither of the subgroups $ \Re(\He^n) $ or $ \Im(\He^n) $ is irreducible. In this article, we  introduce a group $ G_n $ which contains two copies of the Heisenberg group $ \He^n$  as subgroups. We  describe a family of irreducible representations $ \rho_\lambda, \lambda \neq 0 $ and show that the restriction of $ \rho_\lambda $ to each copy of the Heisenberg subgroup is reducible.
\\

The representations $ \rho_\lambda $ of the group $ G_n $ are realised on a Hilbert space of entire functions, namely the twisted Fock spaces. Recall that in the case of $ \He^n $ the Schr\"odinger representation $ \pi_1 $ is unitarily equivalent to the Fock space representation $ \rho_0 $ realised on the Fock space $ \mathcal{F}(\C^n) $ consisting of entire functions on $ \C^n $ that are square integrable with respect to the weight function $ w_0(w) = c_0\, e^{-\frac{1}{2}|w|^2}.$ The unitary operator intertwining $ \pi_1 $ and $ \rho_0 $ is given by the Bargmann transform $ B: L^2(\R^n) \rightarrow \mathcal{F}(\C^n)$ defined by
$$ Bf(z) = e^{-\frac{1}{4}z^2} \, \int_{\R^n} f(\xi) e^{-\frac{1}{2}|\xi|^2}\, e^{z \cdot \xi}\, d\xi.$$
Thus for $ w = u+iv \in \C^n,  \rho_0(w) = B \circ \pi(v,u) \circ B^\ast $ which is explicitly given by
$$ \rho_0(w)F(z) = e^{-\frac{1}{4}|w|^2}\, F(z+w) \, e^{-\frac{1}{2} z\cdot \bar{w}}.$$
As $ \rho_0 $ is irreducible, there are no bounded linear operators on $ \mathcal{F}(\C^n) $ other than constant multiples of the identity that commute with $ \rho_0(w) $ for all $ w \in \C^n.$ The problem of characterising all bounded linear operators $ T $ that commute with $ \rho_0(a), a \in \R^n $ (or with $ \rho_0(ib), b \in \R^n $) has  received some attention  in recent years, see  \cite{CLSWY, CHLS}.
\\

It can be easily seen that any bounded linear operator on the Fock space commuting with all $ \rho_0(a), a \in \R^n $ is of the form  $S_\varphi$ for some $\varphi \in \mathcal{F}(\C^n)$ where
\begin{equation}\label{op-trans} S_\varphi F(z) = \int_{\C^n} F(w) \varphi(z-\bar{w}) e^{\frac{1}{2} z \cdot \bar{w}} \, e^{-\frac{1}{2}|w|^2} dw. 
\end{equation} 
These operators are densely defined and there is a simple necessary and sufficient condition on $ \varphi $ for the operators to be bounded.
Let $ G: L^\infty(\R^n) \rightarrow \mathcal{F}(\C^n) $ be the Gauss-Bargmann transform defined by
\begin{equation}\label{gauss-barg-abelian} 
Gm(z) = e^{\frac{1}{4}z^2} \, \int_{\R^n} m(\xi) e^{i z\cdot \xi}\, e^{-|\xi|^2} \, d\xi.
\end{equation}
Answering a question raised by Zhu \cite{Z},  the authors of the work \cite{CLSWY} have proved that $ S_\varphi $ is bounded on $ \mathcal{F}(\C^n) $ if and only if $\varphi = G m$ for some $ m \in L^\infty(\R^n).$ 
We also have a characterisation of bounded linear operators on the Fock space that commute with $ \rho_0(ib),  b \in \R^n.$ These are of the form
\begin{equation}\label{op-mod}
 \widetilde{S}_\varphi F(z) = \int_{\C^n} F(w) \varphi(z+\bar{w}) e^{\frac{1}{2} z \cdot \bar{w}} \, e^{-\frac{1}{2}|w|^2} dw 
 \end{equation}
for a suitable $ \varphi \in \mathcal{F}(\C^n).$ These two families of operators are related via the identity
$$  \widetilde{S}_\varphi  = U \circ S_{U^\ast \varphi} \circ U^\ast $$
where $ U $ is the unitary operator on $ \mathcal{F}(\C^n) $ defined by $ UF(z) = F(-iz).$
Hence we infer that $ \widetilde{S}_\varphi $ is bounded if and only if $ U^\ast \varphi =G(m) $ for a bounded function $m.$ Some generalisation of this result has been obtained in the recent work \cite{BD}.\\

In an earlier work \cite{Thangavelu-arxiv-Fock-Sobolev-2023} we have proved that for any non constant $ \varphi \in \mathcal{F}(\C^n) $ at least one of the operators $ S_\varphi $ or $ \widetilde{S}_\varphi $ is unbounded. Interestingly, this was proved as a consequence of Hardy's theorem for the Fourier transform on $ \R^n.$
Yet another consequence of these considerations is a new proof of the well known fact that the projective representation $ \rho_0 $ of $ \C^n$ is irreducible. Indeed, if a bounded linear operator $ T $ on $ \mathcal{F}(\C^n) $ commutes with all $ \rho_0(w), w \in \C^n$ then there are functions $ \varphi, \psi \in \mathcal{F}(\C^n) $ such that $ T = S_\varphi = \widetilde{S}_\psi.$ As $ S_\varphi 1 =\varphi$ and $ \widetilde{S}_\psi 1 = \psi $ it follows that $ \varphi = \psi $ and hence $ S_\varphi = \widetilde{S}_\varphi.$ By what we have just mentioned, this is possible only if $ \varphi $ is a constant and hence $ T $ reduces to a constant multiple of the identity. This proves the irreducibility of $ \rho_0.$
\\

We note that in the above proof we haven't made full use of the equality $ S_\varphi = \widetilde{S}_\psi $ but only the fact that $ S_\varphi 1 = \widetilde{S}_\psi 1.$ We can also replace the constant function by the monomials $ \zeta_\alpha(w) = w^\alpha $ for a multi-index $ \alpha.$

\begin{thm}\label{hardy} Suppose $ T $ and $ S $ are bounded linear operators on $ \mathcal{F}(\C^n) $ where $ T $ commutes with $ \rho_0(a), a \in \R^n $ whereas $ S $ commutes with $ \rho_0(ib) $ for all $ b \in \R^n.$ If $ T\zeta_\alpha = S \zeta_\alpha $ for some $ \alpha ,$  then $ T=S=c I $ for some constant $ c.$
\end{thm}

Thus we see that Hardy's theorem is apparently stronger than the irreducibility of $ \rho_0.$  It is not known if they are equivalent, though we believe they are not. However, we can prove the irreducibility of $ \rho_0 $ without recourse to Hardy's theorem. In the following we give another proof of irreducibility which again does not require the full power of the hypothesis.

\begin{thm}\label{weaker-hardy} Suppose $ T $ and $ S $ are as in the previous theorem. If $ T\zeta_\alpha = S \zeta_\alpha $ for all $ |\alpha| \leq 1,$  then $ T= S = c I $ for some constant $ c.$
\end{thm}

We  prove these two theorems in Section 2.1 using some properties of the Gauss-Bargmann transform.
Generalising the Fock spaces, we now consider the following family of Fock spaces $ \mathcal{F}^\lambda(\C^{2n}) $ indexed by $ \lambda \in \R $ which coincide with the classical Fock space $\mathcal{F}(\C^{2n})$ when  $ \lambda =0.$ These spaces are related  to the twisted Bergman spaces $ \mathcal{B}_t^\lambda(\C^{2n}) $ studied in \cite{KTX} which appear naturally in connection with the heat kernel transform or the Segal-Bargmann transform on the Heisenberg group $ \He^n.$ Let us consider the weight function
$$ w_\lambda(z,w) = c_\lambda\, e^{\lambda \Im( z\cdot \bar{w})} \, e^{-\frac{1}{2} \lambda (\coth \lambda) |(z,w)|^2},\,\,\, (z,w) \in \C^{2n}.$$
Then $ \mathcal{F}^\lambda(\C^{2n}) $ is defined as the space of all entire functions $ F $ on $\C^{2n} $ which are square integrable with respect to the measure $ w_\lambda(z,w) \, dz \, dw.$ We equip $ \mathcal{F}^\lambda(\C^{2n}) $ with the norm
$$ \|F \|_{\mathcal{F}^\lambda}^2 = \int_{\C^{2n}} |F(z,w)|^2 \, w_\lambda(z,w) \, dz\, dw. $$ 
Note that when $ \lambda =0, \, w_0(z,w) = c_0\, e^{-\frac{1}{2}  |(z,w)|^2}$ and hence  $ \mathcal{F}^0(\C^{2n}) = \mathcal{F}(\C^{2n}),$ the standard Fock space of entire functions on $ \C^{2n} $ which are square integrable with respect to the Gaussian measure $ d\nu(z,w) = w_0(z,w) \, dz \, dw = c_0\, e^{-\frac{1}{2} |(z,w)|^2} \, dz \, dw.$\\

On the twisted Fock spaces $ \mathcal{F}^\lambda(\C^{2n}) $ we consider the following family of operators $ \rho_\lambda(a,b)$ indexed by  $(a,b) \in \C^{2n}:$
$$
\rho_\lambda(a,b)F(z,w) = e^{\frac{\lambda}{2} \Im (a\cdot \bar{b})}\, e^{-\frac{\lambda}{4} (\coth \lambda)(|a|^2+|b|^2)}\, e^{-i\frac{\lambda}{2}( w\cdot \bar{a}- z \cdot \bar{b})}\, e^{\frac{\lambda}{2} (\coth \lambda)( z\cdot \bar{a}+w \cdot \bar{b})} \, F(z-a,w-b).
$$ 
 It turns out that $ \rho_\lambda(a,b) $ is a unitary operator whose adjoint is given by $ \rho_\lambda(a,b)^\ast = \rho_\lambda(-a,-b).$ Moreover, by a direct calculation we can check that
$$ \rho_\lambda(a,b)\,\rho_\lambda(a^\prime,b^\prime) = e^{i\frac{\lambda}{2}\Re( b \cdot \bar{a^\prime}-a\cdot \bar{b^\prime})}\, e^{-i \frac{\lambda}{2} (\coth \lambda) \Im( a\cdot \bar{a^\prime}+b \cdot \bar{b^\prime})} \,
\rho_\lambda(a+a^\prime, b+b^\prime).$$
When $ \lambda = 0,$ it is well known that $ \rho_0(a,b,t) = e^{-it}\, \rho_0(a,b) $ defines an irreducible unitary representation of the Heisenberg group $ \He^{2n} = \C^n \times \C^n \times \R $ equipped with the group law 
$$(a,b,t)(a^\prime,b^\prime,t^\prime) = (a+a^\prime, b+b^\prime, t+t^\prime+\frac{1}{2}\Im( a\cdot \bar{a^\prime}+b \cdot \bar{b^\prime}))$$
on the Fock space $ \mathcal{F}^0(\C^{2n}).$ But for $ \lambda \neq 0,$ it is not possible to choose a homomorphism $ \varphi_\lambda : \R \rightarrow S^1 $ such that $ \rho_\lambda(a,b,t) = \varphi_\lambda(t) \rho_\lambda(a,b) $ is a representation of $ \He^{2n}$ on $ \mathcal{F}^\lambda(\C^{2n}).$\\

However, the composition formula for $ \rho_\lambda(a,b)\,\rho_\lambda(a^\prime,b^\prime)$ suggests that we equip $ \C^n \times \C^n \times \C $ with the group law
\begin{equation}\label{group law} (a,b,t)(a^\prime,b^\prime,t^\prime) = \big(a+a^\prime, b+b^\prime, t+t^\prime-\frac{\lambda}{2}\Re(  b \cdot \bar{a^\prime}-a\cdot \bar{b^\prime})+\frac{\lambda}{2}(\coth \lambda)\Im( a\cdot \bar{a^\prime}+b \cdot \bar{b^\prime})\big).
\end{equation}
We denote this group by $ \He_\lambda^n(\C) $ and note that  $  \He_\lambda^n(\R) = \R^n \times \R^n \times \R $ is a subgroup of $ \He^n_\lambda(\C) $ which is isomorphic to the real Heisenberg group $ \He^n.$ Indeed, when $ (a,b,t) \in \He^n_\lambda(\R)$ the inherited group law from  $\He_\lambda^n(\C) $ is a deformation of the group law on $ \He^n.$ When $ \lambda =0$ we recover the standard Heisenberg group, $ \He^n_0(\C) = \He^{2n},$  and for this reason we would like to call $ \He^n_\lambda(\C) $ the twisted Heisenberg group. We also observe that $ \R^n \times i\R^n \times \C $ is another subgroup which is isomorphic to $ \C^{n+1}.$ Moreover, by defining 
$$ \rho_\lambda(a,b,t) = e^{-it} \, \rho_\lambda(a,b) $$
we immediately see that $ \rho_\lambda $ defines a unitary representation of $ \He_\lambda^n(\C) $ on $ \mathcal{F}^\lambda(\C^{2n}).$

\begin{thm}\label{irrep} For each $ \lambda \neq 0, \rho_\lambda $ is an irreducible unitary representation of  $ \He_\lambda^n(\C) .$ However, its restriction to $ \He_\lambda^n(\R) $ is not irreducible.
\end{thm}

In a recent work \cite{GT} the authors have characterised all bounded linear operators $ S $ on $ \mathcal{F}^\lambda(\C^{2n}) $ that commute with $ \rho_\lambda(a,b) $ for all $ a, b \in \R^n.$  They have shown that such operators are of the form
\begin{align} \label{def:convolution-operator-tiwsted-fock}
S_\varphi F(z,w) = \int_{\C^{2n}} F(a,b)  \varphi(z-\bar{a}, w-\bar{b}) e^{\frac{1}{2}\lambda (\coth \lambda)(z \cdot \bar{a}+w \cdot \bar{b})} e^{-\frac{i}{2} \lambda (w \cdot \bar{a}- z \cdot \bar{b})} w_\lambda(a,b) \, da \,db  
\end{align}
where $ \varphi \in \mathcal{F}^\lambda(\C^{2n})$  is given by $ \varphi = G_\lambda(M) $ for a bounded linear operator on $ L^2(\R^n).$ Here $ G_\lambda: B(L^2(\R^n)) \rightarrow  \mathcal{F}^\lambda(\C^{2n}) $ is the non-commutative analogue of the Gauss-Bargmann transform  defined by
\begin{align} \label{def:Gauss-Bargmann-transform-tiwsted-fock}
G_\lambda(M)(z,w) = p_1^\lambda(z,w)^{-1}\, \tr\left( \pi_\lambda(-z,-w) e^{-\frac{1}{2} H(\lambda)} M e^{-\frac{1}{2}H(\lambda)} \right) .
\end{align}
Consequently, the ranges of such operators are precisely the subspaces invariant under $ \rho_\lambda(a,b,t), (a,b,t) \in \He^n(\R).$\\

Suppose $ V $ is a closed subspace of $ \mathcal{F}^\lambda(\C^{2n}) $ which is invariant under $ \rho_\lambda(a,b,t), (a,b,t) \in \He^n.$  Then the orthogonal projection $ P_V : \mathcal{F}^\lambda(\C^{2n}) \rightarrow V $ commutes with all $ \rho_\lambda(a,b,t) $ and hence $ P_V = S_{\varphi_V} $ for a unique $ \varphi_V = G_\lambda(M_V) \in \mathcal{F}^\lambda(\C^{2n}).$ It can be shown that the associated  operator $ M_V $ on $ L^2(\R^n) $ is an orthogonal projection.  The problem of identifying all the closed subspaces $ V $ of $ \mathcal{F}^\lambda(\C^{2n}) $ such that $ (V,\rho_\lambda) $ is an irreducible unitary representation of $ \He^n$ is answered in the following result.

\begin{thm}\label{irred-subspace} With notations as above $ (V,\rho_\lambda) $ is irreducible if and only if the projection $ M_V $ is of rank one. Any such $ V $ can be described as $ V = \mathcal{F}^\lambda(\C^{2n}) \ast \varphi $ where $ \varphi = G_\lambda(M_V),$ for a rank one projection $ M_V$ of $ L^2(\R^n).$
\end{thm} 

\begin{rem} Note that $ G_\lambda^\ast(\varphi) $ is a projection if and only if $ \varphi \ast \varphi = \varphi .$  With such a $ \varphi$ it is easy to see that $ \mathcal{F}^\lambda(\C^{2n}) \ast \varphi $ is a closed subspce. Indeed, if $ F_n \ast \varphi \rightarrow F,$ then $ F \ast \varphi = \lim_{n \rightarrow \infty} F_n \ast \varphi \ast \varphi = \lim_{n \rightarrow \infty} F_n \ast \varphi =F.$
\end{rem}

Here is the plan of the paper. After recalling results about the Gauss-Bargmann transform $ G ,$ we define its non-commutative version $ G_\lambda .$ We introduce the operators $ \rho_\lambda(a,b) $ and state relevant results from \cite{GT} regarding bounded linear operators on the twisted Fock space $ \mathcal{F}^\lambda(\C^{2n}).$ In Section 2, we also introduce the algebra $ \mathcal{A}^\lambda(\C^{2n})$.  We introduce the group $ \He^n_\lambda(\C) $ is Section 3 and prove the irreducibility of the representation $ \rho_\lambda.$

\section{Gauss-Bargmann transforms and  some\\
 algebras of entire functions}

\subsection{Gauss-Bargmann transform- the abelian case:} In the classical euclidean setting the Bargmann transform sets up the unitary equivalence between $ L^2(\R^n) $ and $ \mathcal{F}(\C^n).$  It is convenient to consider the Gauss-Bargmann transform rather than the Bargmann transform. Let $ d\gamma(\xi) = e^{-|\xi|^2} d\xi $ and define $ G $ on $ L^2(\R^n, d\gamma)$ by
$$ Gf(z) = e^{\frac{1}{4}z^2} \, \int_{\R^n} f(\xi)  e^{iz \cdot \xi}\, d\gamma(\xi).$$
Note that $ Gf(z) = Bg(iz) $ where $ g(\xi) = f(\xi)e^{-\frac{1}{2}|\xi|^2}$ and hence $ G : L^2(\R^n, d\gamma) \rightarrow \mathcal{F}(\C^n)$ is unitary. As $ L^\infty(\R^n) \subset L^2(\R^n, d\gamma),$ the image of  $ L^\infty(\R^n) $ under $ G ,$ which we denote by $ \mathcal{A}(\C^n),$ is a subspace of $ \mathcal{F}(\C^n).$  We can define a convolution structure  on this subspace  making  it into a commutative Banach algebra.
\\

As proved in \cite{Thangavelu-arxiv-Fock-Sobolev-2023} every $ \varphi \in \mathcal{A}(\C^n) $ defines a bounded linear operator $ S_\varphi .$ By defining $ F \ast \varphi = S_\varphi F $ it has been proved  that $ G^\ast( F \ast \varphi) = G^\ast(F) \, G^\ast(\varphi)$ 
and $ G^\ast(\varphi) = m \in L^\infty(\R^n).$ This shows that $  F \ast \varphi \in \mathcal{A}(\C^n) $ if  both $ F, \varphi \in \mathcal{A}(\C^n).$ We equip the algebra $ \mathcal{A}(\C^n) $ with the norm  defined by 
$$ \| \varphi \| = \| S_\varphi \|_{\text{op}} = \| G^\ast(\varphi)\|_\infty.$$
For $ \varphi, \psi \in \mathcal{A}(\C^n),$ the function $ \varphi \ast \psi $ corresponds to $ S_\varphi \circ S_\psi$ and hence $ \| \varphi \ast \psi \| \leq \|\varphi\| \, \| \psi\|.$ Thus $ \mathcal{A}(\C^n) $ becomes a Banach algebra. The map $ G : L^\infty(\R^n) \rightarrow \mathcal{A}(\C^n) $ is an algebra isomorphism.\\

The case $ \alpha = 0 $ of  Theorem \ref{hardy} was proved in \cite{Thangavelu-arxiv-Fock-Sobolev-2023}. The general case is proved similarly using a stronger version of Hardy's theorem. We omit the details and move on to proving Theorem \ref{weaker-hardy}. \\

{\bf{Proof of Theorem \ref{weaker-hardy}:}} Under the hypothesis of the theorem we have $ T = S_\varphi $ and $ S = \widetilde{S}_\psi.$ The assumption $ T 1 = S1 $ gives $ \varphi = \psi.$  We claim that $ S_\varphi \zeta_j (z)= (-\frac{\partial}{\partial z_j}+ \frac{1}{2} z_j )\varphi(z) $ where $ \zeta_j(w) = \frac{1}{2} w_j.$  It is enough to prove the claim when $ z =a \in \R^n.$ To see this, we start with the observation that the change of variables $ w \rightarrow a-\bar{w} $ in the definition of $ S_\varphi \zeta_j $ yields
$$ S_\varphi \zeta_j(a) = \frac{1}{2} \int_{\C^n} (a_j-\bar{w}_j) \varphi(w)\,e^{\frac{1}{2}a \cdot \bar{w}}\,e^{-\frac{1}{2}|w|^2}\, dw .$$ 
The first term is equal to $\frac{1}{2} a_j \varphi(a) $ whereas the second term is
$$ - \frac{\partial}{\partial a_j} \int_{\C^n}  \varphi(w)\,e^{\frac{1}{2}a \cdot \bar{w}} e^{-\frac{1}{2}|w|^2}\, dw = -\frac{\partial}{\partial a_j}\varphi(a) $$ 
which proves the claim. From  the relation 
$  \widetilde{S}_\varphi  = U \circ S_{U^\ast \varphi} \circ U^\ast $ we also have 
$$ (U^\ast \circ \widetilde{S}_\varphi ) \zeta_j(z) = i (-\frac{\partial}{\partial z_j}+ \frac{1}{2} z_j )U^\ast \varphi(z) = \frac{\partial \varphi}{\partial z_j}(iz) +  \frac{i}{2} z_j \varphi(iz).$$
Thus $ S_\varphi \zeta_j = \widetilde{S}_\varphi \zeta_j $ gives  $(\frac{\partial}{\partial z_j}+ \frac{1}{2} z_j )\varphi(z) = (-\frac{\partial}{\partial z_j}+ \frac{1}{2} z_j )\varphi (z) $ or $ \frac{\partial \varphi}{\partial z_j} = 0 $ for any $ j = 1,2,...,n.$ Consequently, $ \varphi $ reduces to a constant proving the theorem.

\subsection{Gauss-Bargmann transform- non-abelian case:} From now onwards let us assume $ \lambda \neq 0 $ and consider the twisted Fock space $ \mathcal{F}^\lambda(\C^{2n}) $ defined in the introduction. The role of the Gauss-Bargmann transform is now played by its twisted analogue $ G_\lambda $ whose definition requires some preparation. Let $ \pi_\lambda(x,u) = \pi_\lambda(x,u,0) $ where $ \pi_\lambda $ is the Schr\"odinger representation of the Heisenberg group $ \He^n $ corresponding to the central character $ e^{i\lambda t}.$ The representation $ \pi_\lambda(x,u) $ defined for $ (x,u) \in \R^{2n} $ has a natural extension $ \pi_\lambda(z,w) $ to $ \C^{2n}.$  These operators are no longer bounded but  they are densely defined. Let $ H(\lambda) = -\Delta+\lambda^2 |x|^2 $ be the Hermite operator which is self-adjoint with discrete spectrum. We let $ e^{-tH(\lambda)}, t > 0 $  to stand for the Hermite semigroup generated by $ H(\lambda).$ \\

Let $ \mathcal{S}_2(L^2(\R^n)) $ be the space of all Hilbert-Schmidt operators on $ L^2(\R^n).$ We define the Gauss-Bargmann transform $ \mathcal{G}_\lambda $ on $ \mathcal{S}_2(L^2(\R^n)) $ as follows:
$$ \mathcal{G}_\lambda(T)(z,w) = p_1^\lambda(z,w)^{-1}\, \tr \left( \pi_\lambda(-z,-w)Te^{-\frac{1}{2}H(\lambda)} \right).$$
Here $ p_t^\lambda(z,w) $ is the holomorphic extension of the heat kernel associated to the special Hermite operator given explicitly by 
\begin{align} \label{def:heat-kernel-special-Hermite}
p_t^\lambda(y,v) = (4\pi)^{-n} \left( \frac{\lambda}{ \sinh \lambda t}\right)^n e^{-\frac{1}{4}\lambda (\coth \lambda t)(y^2+v^2)}. 
\end{align}
We refer to \cite{KTX} for more about the heat kernel transform associated to the special Hermite operator and the twisted Bergman spaces. In \cite{GT} it has been shown that $ \mathcal{G}_\lambda $ is a unitary operator from $ \mathcal{S}_2(L^2(\R^n)) $ onto $ \mathcal{F}^\lambda(\C^{2n}).$ For any $ M \in B(L^2(\R^n)),$ the operator $ e^{-\frac{1}{2}H(\lambda)} M \in \mathcal{S}_2(L^2(\R^n))$ and hence  the map $ G_\lambda(M) = \mathcal{G}_\lambda(e^{-\frac{1}{2}H(\lambda)} M) $ takes $ B(L^2(\R^n)) $ into a subspace of $ \mathcal{F}^\lambda(\C^{2n}) $ which we denote by $ \mathcal{A}^\lambda(\C^{2n}).$ As in the abelian case we can define a convolution structure on $ \mathcal{A}^\lambda(\C^{2n})$ which will make it into a non-commutative Banach algebra isomorphic to $ B(L^2(\R^n)).$
\\

We can also define the Gauss-Bargmann transform as a map from $ L^2(\R^{2n}) $ onto $ \mathcal{F}^\lambda(\C^{2n}).$ For this we use the Weyl transfom which takes functions  $ f $ on $ \R^{2n} $ into   operators on $ L^2(\R^n)$ defined by 
$$ \pi_\lambda(f) = \int_{\R^{2n}} f(x,y)\, \pi_\lambda(x,y)\, dx \,dy.$$
It is well known that $ \pi_\lambda $ takes $ L^2(\R^{2n}) $ onto $ \mathcal{S}_2(L^2(\R^n)).$ The Plancherel theorem for the Weyl transform  reads as 
$$ (2\pi)^{-n} |\lambda|^{n} \, \int_{\R^{2n}} |f(x,y)|^2 \, dx\,dy =  \|\pi_\lambda(f)\|_{HS}^2.$$ 
We note in passing the interesting relation  $ \pi_\lambda(p_t^\lambda) = e^{-tH(\lambda)}.$ Thus for a suitable constant the map $ \widetilde{\mathcal{G}}_\lambda = d_\lambda \, \mathcal{G}_\lambda \circ \pi_\lambda $ is a unitary operator between $ L^2(\R^{2n}) $ and $ \mathcal{F}^\lambda(\C^{2n}).$ We will make use of this relation in defining the operators $\rho_\lambda.$ For the results mentioned here and for more on Weyl transform we refer to \cite{Book-Thangavelu-uncertainty}.

\subsection{The operators $ \rho_\lambda(a,b)$:} The convolution on the Heisenberg group gives rise to a family of convolutions  for functions on $ \C^n.$ These are known as $\lambda$-twisted convolutions (or simply twisted convolutions) and defined by
$$ f \ast_\lambda g(x,u) = \int_{\R^{2n}} f(x-a, u-b) g(a,b) \,e^{i \frac{\lambda}{2}(u \cdot a-x \cdot b)}\, da\, db.$$
By defining the twisted translation $ \tau_\lambda(a,b) $ of  functions  on $ \R^{2n} $ by
\begin{align} \label{def:twisted-translation} 
\tau_\lambda(a,b)g(x,u) = g(x-a,u-b)\,e^{-i\frac{\lambda}{2}(u\cdot a- x\cdot b)} 
\end{align}
we can write the twisted convolution as 
$$ f \ast_\lambda g(x,u) = \int_{\R^{2n}} f(a, b)\,  \tau_\lambda(a,b)g(x,u)\, da\, db.$$
For easy reference we record some relations between  the two families of unitary operators $ \tau_\lambda(a,b) $ and $ \pi_\lambda(x,u).$ First of all we have
$$ \pi_\lambda(a,b)\pi_\lambda(x,u) = \pi_\lambda(x+a,u+b) e^{-i\frac{\lambda}{2}(u\cdot a- x\cdot b)}.$$
Recalling the definition of $ \pi_\lambda(f) $ another easy calculation shows that
$$ \pi_\lambda(a,b) \pi_\lambda(f) = \int_{\R^{2n}}  f(x,u) \pi_\lambda(x+a,u+b) e^{-i\frac{\lambda}{2}(u\cdot a- x\cdot b)} dx\,du .$$
In other words, the twisted translation satisfies the relation $ \pi_\lambda(a,b) \pi_\lambda(f) = \pi_\lambda( \tau_\lambda(a,b)f) .$ This translates into the property
$ \tau_\lambda(a,b) ( f\ast_\lambda g) = \tau_\lambda(a,b)f \ast_\lambda g $
for  the $\lambda$-twisted translation of $ f $ with $ g .$ \\

By conjugating the unitary operators $ \tau_\lambda(a,b) $ with $ \widetilde{\mathcal{G}}_\lambda, $  we obtain the unitary operators
$$ \widetilde{\rho}_\lambda(a,b) = \widetilde{\mathcal{G}}_\lambda  \circ \tau_\lambda(a,b) \circ \widetilde{\mathcal{G}}_\lambda^\ast  $$
acting on $ \mathcal{F}^\lambda(\C^{2n}).$ If $ F = \mathcal{G}_\lambda(T) $ with $ T = \pi_\lambda(f), f \in L^2(\R^{2n}) $ we see that
$$ \widetilde{\rho}_\lambda(a,b)F(z,w) = e^{\frac{1}{4}\lambda (\coth \lambda)(z^2+w^2)}\, \tr \left( \pi_\lambda(-z,-w) \pi_\lambda(a,b) Te^{-\frac{1}{2}H(\lambda)} \right).$$
In the  above calculation we have used the relation $ \pi_\lambda(a,b) \pi_\lambda(f) = \pi_\lambda( \tau_\lambda(a,b)f) .$ After simplification we obtain the explicit formula
$$ \widetilde{\rho}_\lambda(a,b)F(z,w) = e^{-\frac{\lambda}{4} (\coth \lambda)(|a|^2+|b|^2)}\, e^{-i\frac{\lambda}{2}( w\cdot a- z \cdot b)}\, e^{\frac{\lambda}{2} (\coth \lambda)( z\cdot a+w \cdot b)} \, F(z-a,w-b).$$
Though we have started with $ (a,b) \in \R^{2n} $ we can extend the definition of $ \widetilde{\rho}_\lambda(a,b)  $ for  $ (a,b) \in \C^{2n} $ by setting
\begin{equation}\label{rep-one}
\widetilde{\rho}_\lambda(a,b)F(z,w) = e^{-\frac{\lambda}{4} (\coth \lambda)(|a|^2+|b|^2)}\, e^{-i\frac{\lambda}{2}( w\cdot \bar{a}- z \cdot \bar{b})}\, e^{\frac{\lambda}{2} (\coth \lambda)( z\cdot \bar{a}+w \cdot \bar{b})} \, F(z-a,w-b).
\end{equation} 
By slightly modifying  the representation $ \widetilde{\rho}_\lambda  $ we define
$ \rho_\lambda(a,b) = e^{\frac{\lambda}{2} \Im (a\cdot \bar{b})}\, \widetilde{\rho}_\lambda(a,b) .$ Note that for $ (a,b) \in \R^{2n}, \rho_\lambda(a,b) = \widetilde{\rho}_\lambda(a,b)$ and $\rho_\lambda(ia,ib) = \widetilde{\rho}_\lambda(ia,ib).$ 

\subsection{Operators commuting with $ \rho_\lambda(a,b)$:}

 In an earlier work with Rahul Garg \cite{GT} we studied bounded linear operators on the twisted Fock spaces $\mathcal{F}^\lambda(\C^{2n})$ that commute with $ \rho_\lambda(a,b) $ for all $ (a,b) \in \R^{2n}.$  They turned out to be operators of the form
\begin{align} \label{def:convolution-operator-tiwsted-fock}
S_\varphi F(z,w) = \int_{\C^{2n}} F(a,b)  \varphi(z-\bar{a}, w-\bar{b}) e^{\frac{1}{2}\lambda (\coth \lambda)(z \cdot \bar{a}+w \cdot \bar{b})} e^{-\frac{i}{2} \lambda (w \cdot \bar{a}- z \cdot \bar{b})} w_\lambda(a,b) \, da \,db  
\end{align}
for some $ \varphi \in  \mathcal{F}^\lambda(\C^{2n}).$  These are densely defined operators and in \cite{GT} a necessary and sufficient on $ \varphi $ was found for them to be bounded.  The condition involves  the non-abelian analogue of the Gauss-Bargmann transform $ G $  given by the map $ G_\lambda: B(L^2(\R^n)) \rightarrow  \mathcal{F}^\lambda(\C^{2n}) $ 
\begin{align} \label{def:Gauss-Bargmann-transform-tiwsted-fock}
G_\lambda(M)(z,w) = p_1^\lambda(z,w)^{-1}\, \tr\left( \pi_\lambda(-z,-w) e^{-\frac{1}{2} H(\lambda)} M e^{-\frac{1}{2}H(\lambda)} \right) .
\end{align}

\begin{thm} \label{thm:convolution-operator-twisted-fock}
For any $ \varphi \in \mathcal{F}^\lambda(\C^{2n}) $ consider the operator $ S_\varphi $ defined as in \eqref{def:convolution-operator-tiwsted-fock}. Then it is bounded on $\mathcal{F}^\lambda(\C^{2n}) $ if and only if $ \varphi = G_\lambda(M)$ for some $ M \in B(L^2(\R^n)).$
\end{thm} 

As a corollary to the above theorem the authors obtained  the following representation theorem for the operators $ S_\varphi.$ Let $ \mathcal{G}_\lambda^\ast $ denote the adjoint of $ \mathcal{G}_\lambda.$

\begin{cor} \label{cor:convolution-operator-twisted-fock-via-Gauss-Bargmann-transform} 
For any bounded linear operator $ S_\varphi $ on  $\mathcal{F}^\lambda(\C^{2n})$ let $ M $ be the operator associated to $ \varphi $ as in Theorem \ref{thm:convolution-operator-twisted-fock}. Then
$ S_\varphi F = \mathcal{G}_\lambda \left( \mathcal{G}_\lambda^\ast F  \circ M\right)$ for any $ F \in \mathcal{F}^\lambda(\C^{2n}).$ 
\end{cor}

Along with the operators $ S_\varphi$  the authors in \cite{GT}  also considered  operators of the form $ \widetilde{S}_\varphi $ on  $\mathcal{F}^\lambda(\C^{2n})$  defined by
\begin{align} \label{def:tilde-convolution-operator-tiwsted-fock}
\widetilde{S}_\varphi F(z,w) = \int_{\C^{2n}} F(a,b)  \varphi(z+\bar{a}, w+\bar{b}) e^{\frac{1}{2}\lambda (\coth \lambda)(z \cdot \bar{a}+w \cdot \bar{b})} e^{-\frac{i}{2} \lambda (w \cdot \bar{a}- z \cdot \bar{b})} w_\lambda(a,b) \, da \,db . 
\end{align} 
As in the classical setting, these two classes of operators are related via the unitary operator $ U $ on $ \mathcal{F}^\lambda(\C^{2n})$ defined by $ UF(z,w) = F(-iz,-iw) .$  Indeed, it is easy to see that 
\begin{align} \label{relation:two-convolutions-operators} 
U^\ast \circ S_\varphi \circ U = \widetilde{S}_{U^\ast \varphi}. 
\end{align} 
Since the operators $ S_\varphi$ commute with $ \rho_\lambda(a,b)$ for all $ (a, b)\in \R^{2n},$  in view of the relation \eqref{relation:two-convolutions-operators}, we infer that the operators $ \widetilde{S}_\varphi $ commute with $ U^\ast \circ \rho_\lambda(a,b) \circ U$ for all $ (a,b)\in \R^{2n}.$ An easy calculation shows that for $ (a, b) \in \R^{2n},  U^\ast \circ \rho_\lambda(a,b) \circ U = \rho_\lambda(-ia,-ib) $ and hence $ \widetilde{S}_\varphi $ commutes with $\rho_\lambda(ia,ib)$ for all $ (a, b) \in \R^{2n}.$ \\

As in the classical setting (see \cite{Thangavelu-arxiv-Fock-Sobolev-2023}), there is an uncertainty regarding the simultaneous boundedness of $ S_\varphi $ and $ \widetilde{S}_\varphi.$ 

\begin{thm} \label{thm:uncertainty-twisted-fock-spaces}
For any non constant $ \varphi \in \mathcal{F}^\lambda (\C^{2n})$ at least one of the operators $ S_\varphi $ and $ \widetilde{S}_\varphi $ fails to be bounded on $ \mathcal{F}^\lambda.$
\end{thm} 

The analogue of this theorem in the classical setting was proved by appealing to Hardy's theorem for the Fourier transform on $ \R^n,$ see \cite{Thangavelu-arxiv-Fock-Sobolev-2023} for the proof.  In a similar vein, the above theorem was proved in \cite{GT} by means of Hardy's theorem for the Weyl transform. 
As we have $ U^\ast \circ S_{U \varphi} \circ U = \widetilde{S}_{ \varphi},$ the above theorem can be viewed as the following statement  about the subspace $ \mathcal{A}^\lambda(\C^{2n}) = G_\lambda(B(L^2(\R^n))) \subset \mathcal{F}^\lambda(\C^{2n}) $ under the action of $ U$: if $ \varphi \in \mathcal{A}^\lambda(\C^{2n}),$ then $ U\varphi $ never belongs to $ \mathcal{A}^\lambda(\C^{2n}) $ unless $\varphi$  is a constant.

\subsection{An algebra of entire functions} 
Let $ \mathcal{A^\lambda}(\C^{2n}) $ stand for the subspace of $ \mathcal{F}^\lambda(\C^{2n}) $ consisting of those  $ \varphi $ for which $ S_\varphi $ is bounded. In view of Theorem \ref{thm:convolution-operator-twisted-fock} we know that $ \varphi \in \mathcal{A^\lambda}(\C^{2n}) $ if and only if $ \varphi = G_\lambda (M)$ for some $ M \in B(L^2(\R^n)).$ Thus $\mathcal{A^\lambda}(\C^{2n}) $ is the image of $ B(L^2(\R^n)) $ under $ G_\lambda $ which turns out to be a Banach algebra under a suitable convolution. \\

Let us write $ M = G_\lambda^\ast(\varphi) $ if $ \varphi = G_\lambda(M) $ and note that $ \mathcal{G}_\lambda^\ast(\varphi) = e^{-\frac{1}{2}H(\lambda)} G_\lambda^\ast(\varphi).$ We can rewrite the result of Corollary \ref{cor:convolution-operator-twisted-fock-via-Gauss-Bargmann-transform}  
 as $ \mathcal{G}_\lambda^\ast ( F \ast_\lambda \varphi) = \mathcal{G}_\lambda^\ast (F) \, G_\lambda^\ast(\varphi) .$  It follows that if $ F$ also comes from $ \mathcal{A}^\lambda(\C^{2n}) $ we have the relation 
 $ G_\lambda^\ast(S_\varphi F) = G_\lambda^\ast(F)\, G_\lambda^\ast(\varphi).$ This suggests that by defining  the convolution  $ F \ast_\lambda  \varphi = S_\varphi F,$ we can make the subspace $ \mathcal{A^\lambda}(\C^{2n}) $ into  a non-commutative  algebra. We equip $ \mathcal{A}^\lambda(\C^{2n})$ with the operator norm: $ \| \varphi \| = \|S_\varphi \|_{\text {op}}.$

\begin{prop} For any $ \lambda \neq 0, \, \mathcal{A}^\lambda(\C^{2n}) $ is a non-commutative Banach algebra which is isometrically isomorphic to $ B(L^2(\R^n)).$
\end{prop}
\begin{proof}As $ \varphi \ast_\lambda \psi $ corresponds to the operator $ S_\psi \circ S_\varphi $ it follows that  $ \| \varphi \ast_\lambda \psi \| \leq  \|\varphi\| \| \psi \|.$ The map $ G_\lambda :B(L^2(\R^n)) \rightarrow \mathcal{A}^\lambda(\C^{2n})$ is an isometric isomorphism. Hence the completeness of $ \mathcal{A}^\lambda(\C^{2n}) $ follows from that of $ B(L^2(\R^n)).$
\end{proof}
 
 \begin{rem} Inside the algebra $ \mathcal{A}^\lambda(\C^{2n}) $ there is a commutative subalgebra which is easy to describe. For each bounded function $ m $ on $ \R $ let us consider the bounded operator 
$$ m(H(\lambda)) = \sum_{k=0}^\infty m((2k+n)|\lambda|) \, P_k(\lambda)$$ 
where $ P_k(\lambda) $ are the spectral projections associated to the Hermite operator $ H(\lambda).$ If we define
$$ \mathcal{A}_0^\lambda(\C^{2n}) = \{ \varphi \in \mathcal{A}^\lambda: G_\lambda^\ast(\varphi) =m(H(\lambda)) \} $$
it follows that it is a commutative subalgebra. Moreover, every $ \varphi \in \mathcal{A}_0^\lambda(\C^{2n}) $ is given by the formula
$$ \varphi(z,w) = (2\pi)^{-n} |\lambda|^n \sum_{k=0}^\infty m((2k+n)|\lambda|) \, e^{-2t(2k+n)|\lambda|}\, \varphi_{k,\lambda}^{n-1}(z,w).$$
Here $\varphi_{k,\lambda}^{n-1}(z,w)$ is the holomorphic extension of the Laguerre function
$$\varphi_{k,\lambda}^{n-1}(x,u) = L_k^{n-1}(\frac{1}{2}|\lambda|(|x|^2+u|^2))\, e^{-\frac{1}{4}|\lambda|(|x|^2+|u|^2)}.$$ It would be interesting to study some properties of this subalgebra such as identifying the Gelfand spectrum etc.
\end{rem}

\section{The twisted Heisenberg group $ \He_\lambda^n(\C)$ and \\
the representation $ \rho_\lambda$}  

\subsection{The group $ \He_\lambda^n(\C)$:}  We begin with the definition of $ \rho_\lambda(a,b)$: by slightly modifying 
the operators $ \widetilde{\rho}_\lambda(a,b)  $ we define
$$ \rho_\lambda(a,b)F(z,w) = e^{\frac{\lambda}{2} \Im (a\cdot \bar{b})}\, \widetilde{\rho}_\lambda(a,b)F(z,w) .$$  We record the following properties of $ \rho_\lambda (a,b)$ which can be verified by direct calculation.

\begin{prop} For any $ (a, b), (a^\prime, b^\prime)  \in \C^{2n},$  we have the composition law
$$ \rho_\lambda(a,b)\,\rho_\lambda(a^\prime,b^\prime) = e^{-i\frac{\lambda}{2}\Re( a\cdot \bar{b^\prime}- b \cdot \bar{a^\prime})}\, e^{-i \frac{\lambda}{2} (\coth \lambda) \Im( a\cdot \bar{a^\prime}+b \cdot \bar{b^\prime})} \,
\rho_\lambda(a+a^\prime, b+b^\prime).$$
Moreover, $ \rho_\lambda(a,b)^\ast = \rho_\lambda(-a,-b) $ and hence the operators $ \rho_\lambda(a,b) $ are unitary on $ \mathcal{F}^\lambda(\C^{2n}).$\\
\end{prop}
\begin{proof} It is easy to see that for any $ (a,b) \in \C^{2n}, \widetilde{\rho}_\lambda(a,b) F \in \mathcal{F}^\lambda(\C^{2n}) $ whenever $ F \in \mathcal{F}^\lambda(\C^{2n}).$ Indeed by direct calculation we can verify that 
$$ \int_{\C^{2n}} |\widetilde{\rho}_\lambda(a,b)F(z,w)|^2 \, w_\lambda(z,w)\, dz\, dw = e^{-\lambda \Im(a \cdot \bar{b})} \int_{\C^{2n}} |F(z,w)|^2 w_\lambda(z,w) dz\, dw.$$ In order to find the adjoint,  let us try to prove $ \langle \rho_\lambda(a,b)F, G \rangle =  \langle F, \rho_\lambda(-a,-b) G \rangle $  for   $ F, G \in \mathcal{F}^\lambda(\C^{2n}).$ By defining 
$$ \Theta_{(a,b)}(z,w) =  e^{\frac{1}{2}\lambda (\coth \lambda)(z \cdot \bar{a}+w \cdot \bar{b})} e^{-\frac{i}{2} \lambda (w \cdot \bar{a}- z \cdot \bar{b})} w_\lambda(z,w) $$
we observe that $ \langle \rho_\lambda(a,b)F, G \rangle $ is given by the integral
$$  e^{\frac{1}{2}\Im(a \cdot \bar{b})} e^{-\frac{1}{2}\lambda (\coth \lambda)(|a|^2+|b|^2)} \int_{\C^{2n}} F(z-a,w-b) \Theta_{(a,b)}(z,w)\, \overline{G(z,w)} \,dz \,dw.$$ Our claim will be proved by showing that 
$\Theta_{(a,b)}(z+a,w+b) = \overline{\Theta_{(-a,-b)}(z,w)}.$ Recalling the definition of $ w_\lambda(z,w) $ we calculate that
$$ \Theta_{(a,b)}(z+a,w+b) = e^{-\frac{i}{2} \lambda (w \cdot \bar{a}- z \cdot \bar{b})} \, e^{\lambda \Im(z \cdot \bar{b}-w \cdot \bar{a})} e^{-\frac{1}{2}\lambda (\coth \lambda)(a \cdot \bar{z}+b \cdot \bar{w})}w_\lambda(z,w). $$
The first two factors on the right hand side combine to yield $ e^{-\frac{i}{2}\lambda( \bar{w}\cdot a -\bar{z}\cdot b)}$ which proves the required property of $ \Theta.$
By a simple calculation, it is easy to very that
$$ \widetilde{\rho}_\lambda(a,b)\,\widetilde{\rho}_\lambda(a^\prime,b^\prime) = e^{-i\frac{\lambda}{2}( a\cdot \bar{b^\prime}- b \cdot \bar{a^\prime})}\, e^{-i \frac{\lambda}{2} (\coth \lambda) \Im( a\cdot \bar{a^\prime}+b \cdot \bar{b^\prime})} \,
\widetilde{\rho}_\lambda(a+a^\prime, b+b^\prime).$$
Since $ \rho_\lambda(a,b) = e^{\frac{\lambda}{2} \Im (a\cdot \bar{b})}\, \widetilde{\rho}_\lambda(a,b)$ the stated composition rule for $ \rho_\lambda(a,b) $ follows from that of $ \widetilde{\rho}_\lambda(a,b).$  Finally,  from the composition rule we observe that $ \rho_\lambda(a,b)^\ast \rho_\lambda(a,b) F = F $ and hence $ \rho_\lambda(a,b) $ are unitary.
\end{proof}

The product rule for $ \rho_\lambda(a,b) $ suggests that we equip $ \C^n \times \C^n \times \R $ with the group law (\ref{group law}):
$$(a,b,t)(a^\prime,b^\prime,t^\prime) = (a+a^\prime, b+b^\prime, t+t^\prime-\frac{\lambda}{2}\Re(  b \cdot \bar{a^\prime}-a\cdot \bar{b^\prime})+\frac{\lambda}{2}(\coth \lambda)\Im( a\cdot \bar{a^\prime}+b \cdot \bar{b^\prime})).$$
We denote this group by $ \He_\lambda^n(\C) $ and note that $ \He_\lambda^n = \He_\lambda^n(\R) $ is a subgroup of $ \He_\lambda^n(\C).$  Another subgroup, which is isomorphic to $ \He_\lambda^n $ is given by $ \He_\lambda^n(i\R).$ We extend the the definition of $ \rho_\lambda $ to $ \He_\lambda^n(\C) $ by setting
\begin{equation} \rho_\lambda(a,b,t) = e^{-it} \, \, \rho_\lambda(a,b) .
\end{equation}
It is then easy to see that $ \rho_\lambda $ is a unitary representation of $ \He_\lambda^n(\C) $ on the twisted Fock space.

\subsection{The irreducibility of $ \rho_\lambda$:} By using the results stated in Section 2.4 and proved in \cite{GT}, we can easily see that  the representation $ \rho_\lambda $ of $ \He_\lambda^n(\C) $ is  irreducible. Indeed, if a bounded linear operator $ T $ on $ \mathcal{F}^\lambda(\C^{2n})$ commutes with $ \rho_\lambda(g) $ for all $ g \in \He_\lambda^n(\C)$ then it will commute with $ \rho_\lambda(a,b) $ for all $ (a,b) \in \R^{2n} $ and hence  $ T = S_\varphi $ defined in \ref{def:convolution-operator-tiwsted-fock} for some $ \varphi $ satisfying the condition given in Theorem \ref{thm:convolution-operator-twisted-fock}. From Corollary \ref{cor:convolution-operator-twisted-fock-via-Gauss-Bargmann-transform} we infer that $ T 1= S_\varphi 1 = \varphi.$ As $ T $ commutes with $ \rho_\lambda(ia,ib), (a,b) \in \R^{2n}$ we also know that $ T = \widetilde{S}_\psi.$ In view of \ref{relation:two-convolutions-operators} we also infer that $ T 1 = \widetilde{S}_\psi 1 = \psi.$ And hence $ \varphi = \psi $ and both operators $ S_\varphi $ and $ \widetilde{S}_\varphi $ are bounded. But then by Theorem \ref{thm:uncertainty-twisted-fock-spaces} this is possible only if $ \varphi $ is a constant which means $ T = c I$ proving the irreducibility of $ \rho.$\\

But the above proof depends on Hardy's theorem for the Weyl transform which is much stronger than the irreducibility of $ \rho_\lambda.$ In what follows we give a direct proof following the same line of arguments  used in proving Theorem \ref{thm:uncertainty-twisted-fock-spaces}\\

\noindent{\bf{Proof of Theorem \ref{irrep}:}} As above assuming that  a bounded linear operator $ T $ commutes with all $ \rho_\lambda(g) $ we conclude that there exists $ \varphi = G_\lambda(M) $ such that $ S_\varphi = T = \widetilde{S}_\varphi.$
Recall that
$$S_\varphi F(z,w) = \int_{\C^{2n}} F(a,b)  \varphi(z-\bar{a}, w-\bar{b}) e^{\frac{1}{2}\lambda (\coth \lambda)(z \cdot \bar{a}+w \cdot \bar{b})} e^{-\frac{i}{2} \lambda (w \cdot \bar{a}- z \cdot \bar{b})} w_\lambda(a,b) \, da \,db . $$
If we assume that $ z, w \in \R^n, $ then we can make a change of variables to verify that
$$S_\varphi F(z,w) = \int_{\C^{2n}} F(z-\bar{a}, w-\bar{b}) \varphi(a,b)\, e^{\frac{1}{2}\lambda (\coth \lambda)(z \cdot \bar{a}+w \cdot \bar{b})} e^{\frac{i}{2} \lambda (w \cdot \bar{a}- z \cdot \bar{b})} w_{-\lambda}(a,b) \, da \,db . $$
Since we already know that $ S_\varphi 1 = \varphi,$ the above equation shows that
\begin{equation}\label{reproduce}
 \varphi(z,w) = \int_{\C^{2n}}  \varphi(a,b)\, e^{\frac{1}{2}\lambda (\coth \lambda)(z \cdot \bar{a}+w \cdot \bar{b})} e^{\frac{i}{2} \lambda (w \cdot \bar{a}- z \cdot \bar{b})} w_{-\lambda}(a,b) \, da \,db  
 \end{equation}
which is valid for all $ z, w \in \C^n $ by analytic continuation. For $ j=1,2,...,n $ let us define 
$$ \zeta_{j,0}(z,w) = \frac{1}{2}\lambda \left( (\coth \lambda) z_j -i w_j \right),\,\,\,  \zeta_{0,j}(z,w) = \frac{1}{2}\lambda \left( (\coth \lambda) w_j +i z_j \right). $$
From (\ref{reproduce}) we obtain the following relations:
$$ S_\varphi \zeta_{j,0} = -\frac{\partial \varphi}{\partial z_j}+\zeta_{j,0} \, \varphi,\,\, S_\varphi \zeta_{0,j} = -\frac{\partial \varphi}{\partial w_j}+\zeta_{0,j}\, \varphi .$$
Using the relation $ U^\ast \circ S_\varphi \circ U = \widetilde{S}_{U^\ast \varphi}$ we also have
$$ \widetilde{S}_\varphi \zeta_{j,0} = \frac{\partial \varphi}{\partial z_j}+\zeta_{j,0} \, \varphi,\,\, \widetilde{S}_\varphi \zeta_{0,j} = \frac{\partial \varphi}{\partial w_j}+\zeta_{0,j}\, \varphi .$$
Since $ S_\varphi  = \widetilde{S}_\varphi,$ from the above we conclude that $\frac{\partial \varphi}{\partial z_j}=\frac{\partial \varphi}{\partial w_j} =0 $ for all $j $ and hence $ \varphi$ reduces to a constant.
This proves that $ \rho_\lambda $ is irreducible.

\begin{rem} Note that the above proof does not require the full power of the equality $ S_\varphi  = \widetilde{S}_\varphi $ but uses only the fact that these two operators agree with all monomials of degree one.
\end{rem}

\subsection{The restriction of $ \rho_\lambda $ to $ \He_\lambda^n$:} Though $ \rho_\lambda $ is an irreducible unitary representation of $ \He_\lambda^n(\C) ,$ it turns out that its restriction to $ \He_\lambda^n(\R) $ is not irreducible. Note that for $ (a,b,t) \in \He_\lambda^n(\R), \rho_\lambda(a,b,t) = e^{-i t} \rho_\lambda(a,b).$ From Theorem \ref{thm:convolution-operator-twisted-fock} we know that operators of the form $ S_\varphi $ with $ \varphi = G_\lambda(M), M \in B(L^2(\R^n)) $ commute with $ \rho_\lambda(a,b)$ for all $ a, b \in \R^n.$ Consequently range of such operators are invariant under $ \rho_\lambda(g), g \in \He^n_\lambda(\R).$ It is possible to choose $ \varphi $ so that the closure of the range of $ S_\varphi$ is a proper subspace of $ \mathcal{F}^\lambda(\C^{2n}) $ and hence $ \rho_\lambda $ is not irreducible.\\

In what follows we try to characterise closed subspaces of $ \mathcal{F}^\lambda(\C^{2n}) $ on which $ \rho_\lambda $ becomes an irreducible unitary representation of $ \He^n.$ Let $ V \subset \mathcal{F}^\lambda(\C^{2n}) $ be a closed subspace invariant under $ \rho_\lambda(g), g \in \He^n.$ Then the orthogonal projection $ P_V :\mathcal{F}^\lambda(\C^{2n}) \rightarrow V $ is a bounded linear operator which commutes with $ \rho_\lambda(a,b), a, b \in \R^n.$ Then there exists a unique $ \varphi_V \in \mathcal{A}^\lambda(\C^{2n}) $ and $ M_V \in B(L^2(\R^n))$ such that $ \varphi_V = G_\lambda(M_V) $ and $ S_{\varphi_V}= P_V.$ As $ P_V^2 = P_V,$ it follows that $ \varphi_V \ast_\lambda \varphi_V = \varphi_V $ and hence $ M_V= M_V^2 $ is an orthogonal projection of $L^2(\R^n).$ \\

\noindent{\bf{Proof of Theorem \ref{irred-subspace}:}} Let $  V $ be a closed subspace of $ \mathcal{F}^\lambda(\C^{2n}) $ which supports the representation $ \rho_\lambda.$  Let $ V = W \oplus W^\prime $ be the direct sum of two closed subspaces both invariant under $ \rho_\lambda(g), g \in \He^n.$ As $ P_V = P_W \oplus P_{W^\prime} $ it follows that $ \varphi_V = \varphi_W + \varphi_{W^\prime} $ and $ \varphi_W \ast_\lambda  \varphi_{W^\prime} = 0 .$ This in turn implies that $ M_V = M_W + M_{W^\prime} $ and $ M_W \circ M_{W^\prime} =0.$ 
As a consequence we infer that $ V $ has a proper invariant subspace if and only of the projection $ M_V $ associated to $ V $ has a decomposition of the form $ M_V = M_W \oplus M_{W^\prime}.$ This proves the theorem.

\begin{rem} From the above result we see that if  the projection $ M_V $ associated to $ V $ is of rank $ k $ then we can decompose $ V = \oplus_{j=1}^k V_j $ and each $ (V_j,\rho_\lambda) $ is irreducible. Given $ M \in B(L^2(\R^n)) $ with a spectral decomposition $ M = \sum_{j=0}^\infty \lambda_j\, E_j $ where the spectral projections are of finite rank, say $ d_j.$  For example, any compact normal operator $ M $ provides such an example. Then  with $ \varphi_j =G_\lambda(E_j) $ the subspaces $ V_j = \mathcal{F}^\lambda(\C^{2n}) \ast_\lambda \varphi_j ,$ the range of $ S_{\varphi_j} ,$ supports a primary representation which can be further decomposed  into finitely many irreducible subspaces.
\end{rem}



\section*{Acknowledgements}
This work was carried out at the University of Queensland, Brisbane. He wishes to thank the Department of Mathematics for the hospitality and  Prof. Dietmar Oelz for the invitation. He also wishes to thank Rahul Garg for several useful discussions on the subject matter of this article. This work was supported in parts by DST (J. C. Bose Fellowship) and INSA \\

\providecommand{\bysame}{\leavevmode\hbox to3em{\hrulefill}\thinspace}
\providecommand{\MR}{\relax\ifhmode\unskip\space\fi MR }
\providecommand{\MRhref}[2]{%
  \href{http://www.ams.org/mathscinet-getitem?mr=#1}{#2}
}
\providecommand{\href}[2]{#2}

\end{document}